\documentclass[12pt]{amsart}
\usepackage{amsmath}
\usepackage{amsfonts}
\usepackage{amsthm}
\usepackage{amssymb}

\theoremstyle{plain}
\numberwithin{equation}{section}
\newtheorem{Theorem}{Theorem}
\newtheorem{Lemma}[Theorem]{Lemma}
\newtheorem{Proposition}[Theorem]{Proposition}
\newtheorem{Corollary}[Theorem]{Corollary}

\newtheorem{Criterion}[Theorem]{Criterion}
\newtheorem{Example}[Theorem]{Example}

\begin{document}

\title[Riesz basis property of Hill operators]{Riesz basis property
of Hill operators with potentials in weighted spaces}

{\author{Plamen Djakov}}\thanks{P. Djakov acknowledges the hospitality
of the Department of Mathematics of the Ohio State University,
July--August 2013.}

\author{Boris Mityagin}\thanks{B. Mityagin acknowledges the support
of the Scientific and Technological Research Council of Turkey and
the hospitality of Sabanci University, May--June, 2013.}

\dedicatory{Dedicated to the memory of Boris Moiseevich Levitan \\on
the occasion of the 100th anniversary of his birthday}

\address{Sabanci University, Orhanli,
34956 Tuzla, Istanbul, Turkey}
 \email{djakov@sabanciuniv.edu}

\address{Department of Mathematics,
The Ohio State University,
 231 West 18th Ave,
Columbus, OH 43210, USA} \email{mityagin.1@osu.edu}

\begin{abstract}
Consider the Hill operator $L(v) = - d^2/dx^2 + v(x) $  on  $[0,\pi]$
with Dirichlet, periodic or antiperiodic boundary conditions; then
for large enough $n$  close
to $n^2 $ there are one Dirichlet eigenvalue $\mu_n$ and two
periodic (if $n$ is even) or antiperiodic (if $n$ is odd) eigenvalues
$\lambda_n^-, \, \lambda_n^+ $ (counted with multiplicity).

 We describe classes of complex potentials
$v(x)= \sum_{2\mathbb{Z}} V(k) e^{ikx}$  in weighted spaces  (defined
in terms of the Fourier coefficients of $v$) such that the periodic
(or antiperiodic) root function system of $L(v) $ contains a Riesz
basis if and only if
$$V(-2n) \asymp V(2n) \quad \text{as} \;\; n \in 2\mathbb{N}\;\;
(\text{or} \; n \in 1+ 2\mathbb{N}), \;\; n \to \infty.$$ For such
potentials we prove that $\lambda_n^+ - \lambda_n^- \sim \pm
2\sqrt{V(-2n)V(2n)} $ and
$$\mu_n - \frac{1}{2}(\lambda_n^+ + \lambda_n^-) \sim -\frac{1}{2}
(V(-2n) + V(2n)).$$

\end{abstract}

\maketitle



\section{Introduction}
The theory of self-adjoint ordinary differential operators (o.d.o.)
is well-developed, and the spectral decompositions play a central
role in it \cite{LS, Na69, Mar}.

Convergence of the spectral decompositions of non-self-adjoint
o.d.o., considered on a finite interval $I$ and subject to strictly
regular boundary conditions (see \cite[\S 4.8]{Na69}), has been
understood completely in the early 1960's \cite{Mi62, Ke64, DS71}. In
this case, we not only have convergence, but the system of
eigenfunctions (SEF) is a Riesz basis in $L^2 (I).$ However, in the
case of regular but not strictly regular boundary conditions -- even
in the case of periodic or antiperiodic  boundary conditions --
complete understanding appeared only in the  2000's as a result of
the interaction of two lines of research.

One stems out from a question raised by A. Shkalikov in 1996/1997 in
Kostyuchenko-Shkalikov seminar on Spectral Analysis at Moscow State
University. He formulated the following assertion and sketched an
approach to its proof.

{\em Consider the Hill operator
$$
Ly=-y''+ q(x)y, \quad   0 \leq x \leq \pi,
$$
with a smooth potential $q$ such that for some $s\geq 0$
$$
q^{(k)} (0) = q^{(k)} (\pi), \quad 0 \leq k \leq s-1,
$$
and
$$
q^{(s)} (0) - q^{(s)} (\pi) \neq 0.
$$
Then the system of normalized periodic (or antiperiodic) root functions
of the operator $L = L(q) $ is a Riesz basis in $L^2 ([0,\pi]).$}

In the framework of the scheme suggested by Shkalikov, this claim
was proved in the case $q \in C^4 ([0,\pi]), \; s=0,$ by Kerimov and
Mamedov \cite{MK}. Further results of Dernek-Veliev \cite{DV05},
Makin \cite{Ma06-1} and Veliev-Shkalikov \cite{VS09}  confirmed the
general case $s \geq 0.$ Moreover,  Makin \cite{Ma06-1}  considered
potentials $q(x) = \sum_{k\in 2\mathbb{Z}} q_k e^{ikx}$ such that
\begin{equation}
\label{m1} q \in W^s_1 ([0,\pi]), \quad  q^{(p)}(0) = q^{(p)}(\pi),
\;\; 0 \leq p \leq s-1,
\end{equation}
and
\begin{equation}
\label{m2} \exists c>0: \quad |q_{\pm 2n}| > c \, n^{-s-1} \quad
 \forall n>>1,
\end{equation}
and proved that \\
{\em the periodic (antiperiodic)
SEF is a Riesz basis if and only if   there are constants $C>c>0$
such that}
\begin{equation}
\label{m3}
  c \, q_{-2n}  \leq q_{2n}  \leq   C \, q_{-2n}, \;\;  \forall \;
  \text{even (odd)} \;  n>>1.
\end{equation}
He used this result to construct examples of potentials
for which the periodic SEF is not a Riesz basis.
Veliev and Shkalikov \cite{VS09} extended the results of Makin
by providing more general conditions for existence of
Riesz bases.

Notice, however, that the above results were obtained for
potentials of finite smoothness, i.e., in the framework of
Sobolev spaces $W_1^m, $ where $m$ is a positive integer.

Another line of research comes from the papers
 \cite{KM99, KM01, DM3, DM4, DM5} which goal was the analysis
 of spectral gaps $\gamma_n = \lambda_n^+ - \lambda_n^-$
and deviations $\delta_n = \mu_n - \frac{1}{2} (\lambda_n^+ +
\lambda_n^-),$  but the analytical methods developed
in these papers
 allowed us to understand the structures responsible for the
 Riesz basis property of SEF in the case of periodic or
 antiperiodic boundary conditions.
Already in \cite{Mit03} it has been announced that the authors have
constructed examples of 1D Dirac operators such that their periodic
or antiperiodic SEF is not a Riesz basis. With all details
these constructions were presented, both for Hill and
1D Dirac operators in \cite[Section~5.2]{DM15},
see in particular Theorem 71 there.

Recently, the same approach has led to general
necessary and sufficient conditions
for existence of Riesz bases consisting of
periodic (or antiperiodic) root functions
\cite{DM25a, DM25, DM28}.

This note  gives  a further development of those results in the
framework of the approach in \cite[Section 5.2]{DM15} (see
Theorems~9 and 10 below).  We work in weighted spaces of potentials,
which allows us to consider potentials of arbitrary smoothness
(including singular potentials or potentials of smoothness beyond
$C^\infty,$ say in the Carlemann-Gevrey  classes).

We describe classes of complex potentials $v(x)= \sum_{2\mathbb{Z}}
V(k) e^{ikx}$  (in weighted spaces defined in terms of the Fourier
coefficients $V(k)$ of $v$) such that the periodic or antiperiodic
root function system of the Hill operator $L(v) $ contains a Riesz
basis if and only if
$$V(-2n) \asymp V(2n) \quad \text{as} \;\; n \in 2\mathbb{N}\;\;
(\text{or} \; n \in 1+ 2\mathbb{N}), \;\; n \to \infty.$$ For such
potentials we prove that $\lambda_n^+ - \lambda_n^- \sim \pm
2\sqrt{V(-2n)V(2n)}  $ and $$\mu_n - \frac{1}{2}(\lambda_n^+ +
\lambda_n^-) \sim -\frac{1}{2} (V(-2n) + V(2n)).$$

Moreover, we give several examples (Section 5) to illustrate our
main statements, where we overcome additional difficulties when
verifying the general conditions (on asymptotics of crucial
sequences $\beta_n^\pm $  -- see Propositions~\ref{ex1}, \ref{ex2}
and other claims in Section 5).

\section{Preliminaries}

Let $L= L(v)$ be  the Hill operator
\begin{equation}
\label{i1} L y=-y^{\prime \prime}+v(x)y,
\end{equation}
with a complex valued potential $v \in L^2 ([0,\pi]$ or more
generally, with a singular complex valued potential $v \in
H^{-1}_{per} (\mathbb{R}).$

For potentials $v \in L^2 ([0,\pi]$ we consider $L(v) $   on the
interval $[0,\pi]) $ with Dirichlet $(Dir),$  periodic  $(Per^+)$
and antiperiodic  $(Per^-)$ boundary conditions $(bc):$
\begin{eqnarray}
\label{i3}  Dir: \quad  y(0) = 0, \quad y(\pi) = 0, \\
\label{i4} Per^\pm: \quad   y(\pi) = \pm y(0), \quad y^\prime (\pi) =
\pm y^\prime (0).
\end{eqnarray}
Singular $\pi$-periodic potentials  $v \in H^{-1}_{per} (\mathbb{R})
$ have the form $v= C + Q^\prime,$ where $C$ is a constant and $Q$
is a $\pi$-periodic function such that $Q \in L^2_{loc}
(\mathbb{R}).$  Since adding a constant results in a shift of the
spectra, we may consider without loss of generality only
$\pi$-periodic potentials of the form
\begin{equation}
\label{i01}
v(x) = Q^\prime  (x), \quad
 Q \in L^2_{loc}(\mathbb{R}), \;\; Q(x + \pi) = Q(x).
\end{equation}
Let us notice that if $v$ is a $\pi$-periodic function with $v \in
L^1_{loc} (\mathbb{R}),$ then it has the form \eqref{i01} if, and
only if,
\begin{equation}
\label{i00} \int_0^\pi v(t) dt = 0
\end{equation}
because the latter condition implies that $Q(x)= \int_0^x v(t) dt$ is
a $\pi$-periodic function.

In the case of potentials $v \in H^{-1}_{per} (\mathbb{R})$ the
classical periodic and antiperiodic boundary conditions \eqref{i4}
are replaced by
\begin{equation}
\label{i04}
 y(\pi) = \pm y(0), \quad y^{[1]} (\pi) =
\pm y^{[1]} (0),
\end{equation}
where
$$
y^{[1]}(x) := y^\prime (x) - Q(x) y(x)
$$
is the {\em quasi-derivative} of $y.$ We refer to \cite{SS03, HM01,
DM16,DM21} for basics and details about Hill-Schr\"odingier
operators with singular potentials of the form (\ref{i01}). The
Fourier method for such operators is developed in \cite{DM16}. We
recall that the Fourier coefficients of $v$ with respect to the
orthonormal system $\left (e^{ikx}\right )_{k\in 2Z}$ are defined by
\begin{equation}
\label{i05} V(k) =ik q(k), \quad \text{where} \;\;  q(k) =
\frac{1}{\pi} \int_0^\pi Q(x) e^{ikx} dx,  \;\; k \in 2Z.
\end{equation}

It is known (see \cite{Mar, DM15}  for $L^2$-potentials, or
\cite{DM16, DM21} for $H^{-1}_{per} $-potentials)  that the
following holds.
\begin{Lemma}
\label{loc} Let $v$ be a potential of the form \eqref{i01}. Then the
periodic, antiperiodic and Dirichlet spectra of the operator $L(v)$
are discrete. Moreover, there is an integer $N_* = N_* (v)$  such
that for each $n>N_* $ the disc
\begin{equation}
\label{i06} D_n = \{\lambda \in \mathbb{C}:  \; |\lambda - n^2 | <
n/4\}
\end{equation}
contains one simple Dirichlet eigenvalue and two periodic (if $n$ is
even) or antiperiodic (if $n$ is odd) eigenvalues $\lambda_n^-, \,
\lambda_n^+ $ (counted with multiplicity). There are at most
finitely many periodic, antiperiodic and Dirichlet eigenvalues
outside the union  $\bigcup_{n\geq N_*} D_n,$  and that eigenvalues
are situated in the half-plane $Re \, z < (N_* +1/2)^2.$
\end{Lemma}

The smoothness of potentials $v$ can be characterized in terms of
decay rate of the {\em spectral gaps} $\gamma_n = \lambda_n^+  -
\lambda_n^-$ and {\em deviations} $\delta_n = \mu_n - \lambda_n^+$
(see \cite{DM15} and the bibliography therein for Hill operators with
$L^2$-potentials, and \cite{DM21} for Hill operators with singular
potentials). The proofs of these results use essentially the
following statement (see \cite[Section~2.2]{DM15} for Hill operators
with $L^2$-potentials and \cite[Lemma 6]{DM21} for Hill-Schr\"odinger
operators with $H^{-1}_{per}$-potentials).

\begin{Lemma}
\label{basic} There are functionals $\alpha_n (v;z) $ and $
\beta^\pm_n (v;z) $ defined for large enough $n \in \mathbb{N}$ and $
|z| < n $ such that $\lambda = n^2 + z $  is a periodic (for even
$n$) or antiperiodic (for odd $n$) eigenvalue of $L$ if and only if
$z$ is an eigenvalue of the matrix
\begin{equation}
\label{p1}  \left [
\begin{array}{cc} \alpha_n (v;z)  & \beta^-_n (v;z)
\\ \beta^+_n (v;z) &  \alpha_n (v;z) \end{array}
\right ].
\end{equation}
Moreover, $\alpha_n (z;v) $ and $\beta^\pm_n (z;v)$ depend
analytically on $v$ and $z,$  and $z_n^\pm =\lambda_n^\pm -n^2$
 are the only solutions of the basic equation
\begin{equation}
\label{p2}   (z-\alpha_n (v;z))^2=  \beta^-_n (v;z)\beta^+_n (v;z)
\end{equation}
in the disc $ |z| < n/4. $
\end{Lemma}

The functionals $\alpha_n (v;z) $ and $\beta^\pm_n (v;z)$ are well
defined for large enough $n$ by the following expressions in terms
of the Fourier coefficients of the potential (see (2.16)--(2.33) in
\cite{DM15} for Hill operators with $L^2$-potentials and
(3.23)--(3.30) in \cite{DM21} for Hill operators with
$H^{-1}_{per}$-potentials).

\begin{equation}
\label{22.9} \alpha_k  =  \sum_{k=1}^\infty  S^{11}_k, \quad
\beta_n^- = V(-2n) + \sum_{k=1}^\infty  S^{12}_k, \quad
\beta_n^+ =V(2n) + \sum_{k=1}^\infty  S^{21}_k,
\end{equation}
where for $k=1,2, \dots$
\begin{equation}
\label{22.11} S^{11}_k =
 \sum_{j_1, \ldots, j_k \neq \pm n}
\frac{V(-n-j_1)V(j_1 - j_2) \cdots V(j_{k-1} -j_k) V(j_k +n)}
{(n^2 -j_1^2 +z) \cdots (n^2 - j_k^2 +z)},
\end{equation}
and
\begin{equation}
\label{22.14} S^{12}_k =  \sum_{j_1, \ldots, j_k \neq \pm n}
\frac{V(-n-j_1)V(j_1 - j_2) \cdots V(j_{k-1} -j_k) V(j_k -n)}
{(n^2 -j_1^2 +z) \cdots (n^2 - j_k^2 +z)},
\end{equation}
\begin{equation}
\label{22.15} S^{21}_k = \sum_{j_1, \ldots, j_k \neq \pm n}
\frac{V(n-j_1)V(j_1 - j_2) \cdots V(j_{k-1} -j_k) V(j_k +n)} {(n^2
-j_1^2 +z) \cdots (n^2 - j_k^2 +z)}.
\end{equation}
In the sequel, we suppress the dependence on $v$ in the notations
and write only $\beta^{\pm}_n (z), \alpha_n (z).$

\begin{Lemma}
\label{lemrho} If $v$ is a singular potential of the form
\eqref{i01}, and  $\lambda_n^\pm $ are the corresponding periodic or
antiperiodic eigenvalues in the disc $D_n, $ then
\begin{equation}
\label{i08} |\lambda_n^\pm -n^2| =  o (n), \quad n \to \infty.
\end{equation}
\end{Lemma}

\begin{proof}
In view of  \cite[(4.32)]{DM21},
$$
|\alpha_n (z)| \leq n \varepsilon_n, \quad |\beta_n^\pm (z)| \leq n
\varepsilon_n + |V(\pm 2n)| \quad \text{for} \;\; |z| \leq n/2,
$$
where $V(k)$ are given by (\ref{i05}) and
\begin{equation}
\label{r1} \varepsilon_n: = C_1 \left (\sum_{|k|\geq \sqrt{n}}
|q(k)|^2 \right )^{1/2} + \frac{C_2}{\sqrt{n}}
\end{equation}
with some constants  $C_1, C_2.  $  Therefore, $\varepsilon_n \to 0$
as $n\to \infty. $ On the other hand, by (\ref{i05}) we have $V(\pm
2n) = \pm 2in q(\pm 2n)$ with $q(\pm 2n) \to 0.$

Since $z_n^\pm = \lambda_n^\pm -n^2 $  are roots of (\ref{p2}), it
follows that
\begin{equation}
\label{r2} |z_n^\pm|/n  \leq   \varepsilon_n + \sqrt{(\varepsilon_n
+ 2|q(-2n)|)(\varepsilon_n + 2|q(2n)|)} \to 0.
\end{equation}
Thus, (\ref{i08}) holds.
\end{proof}

{\em Remark. The estimate in \eqref{i08} could be improved if $v(x) =
Q^\prime (x) $ with  $Q \in H^\alpha, \; 0< \alpha <1. $ Then one can
show that}
\begin{equation}
\label{i09} |\lambda_n^\pm -n^2| =  o (n^{1-\alpha}), \quad n \to
\infty.
\end{equation}
\bigskip

The asymptotic behavior of $\beta^{\pm}_n (z)$ (or $\gamma_n$ and
$\delta_n$) plays also a crucial role in studying the Riesz
basis property of the system of root functions of the operators
$L_{Per^\pm}.$ In \cite[Section 5.2]{DM15}, it is shown (for
potentials $v \in L^2 ([0,\pi])$) that if the ratio $\beta^+_n
(z_n^*)/\beta^-_n (z_n^*) $ is not separated from $0$ or $\infty$
then the system of root functions of $L_{Per^\pm}$ does not contain
a Riesz basis (see Theorem 71 and its proof therein). Theorem 1 in
\cite{DM25} (or Theorem 2 in \cite{DM25a}) gives, for wide classes
of $L^2$-potentials,  the following criterion for Riesz basis
property.

\begin{Criterion}
\label{crit1}  Consider the Hill operator with  $v \in L^2
([0,\pi]). $ If
\begin{equation}
\label{a1} \beta_n^+ (0) \neq 0, \quad \beta_n^- (0)\neq 0
\end{equation}
and
\begin{equation}
\label{a2} \exists \, c  \geq 1 : \quad  c^{-1}|\beta_n^\pm (0)|
\leq |\beta_n^\pm (z)| \leq c \, |\beta_n^\pm (0)|, \quad  |z| \leq
1,
\end{equation}
for all sufficiently large even $n$ (if $bc= Per^+$) or odd $n$ (if
$bc= Per^-$), then

(a) there is $N=N(v) $ such that for $n>N$ the operator
$L_{Per^\pm}(v)$ has exactly two simple periodic (for even $n$) or
antiperiodic (for odd $n$) eigenvalues in the disc $\{z: |z- n^2|<1
\} ;$

(b)  the system of root functions of $L_{Per^+} (v) $ or $L_{Per^-}
(v) $ contains a Riesz basis in $L^2 ([0,\pi])$ if and only if,
respectively,
\begin{equation}
\label{a3} \limsup_{n\in 2\mathbb{N}} t_n (0) <\infty \quad
\text{or}
 \quad \limsup_{n\in 1+2\mathbb{N}} t_n (0)<\infty,
\end{equation}
where
\begin{equation}
\label{a4} t_n (z)=\max \{|\beta_n^-(z)|/|\beta_n^+(z)|, \,
|\beta_n^+(z)|/|\beta_n^-(z)|\}.
\end{equation}
\end{Criterion}

In general form, i.e., without the restrictions (\ref{a1}) and
(\ref{a2}), this criterion is given in \cite{DM26-2} in the context
of 1D Dirac operators but in the case of Hill operators the
formulation and the proof are the same (see Proposition~19 in
\cite{DM28}). Moreover, the same argument gives the following more
general statement.

\begin{Criterion}
\label{crit2} Let $\Gamma^+ =2\mathbb{N}, $ $\Gamma^-=2\mathbb{N}-1 $
in the case of Hill operators, and $\Gamma^+ =2\mathbb{Z}, $
$\Gamma^-=2\mathbb{Z}-1 $ in the case of one dimensional Dirac
operators. There exists $N_* = N_* (v)$ such that for $|n|>N_*$ the
operator $L=L_{Per^\pm}(v)$ has in the disc $D_n =\{z: |z- n^2|<n/4
\}$  (respectively $D_n =\{z: |z- n|<1/2 \}$) exactly two periodic
(for $n\in \Gamma^+$) or antiperiodic (for $n\in \Gamma^-$)
eigenvalues, counted with algebraic multiplicity. Let
    $$\mathcal{M}^\pm =\{n\in \Gamma^\pm: \; |n|
\geq N_*, \; \lambda^-_n \neq \lambda^+_n \},$$ and  let $\{u_{2n-1},
\,u_{2n} \}$ be a pair of normalized eigenfunctions associated, respectively,  with
the eigenvalues $\lambda_n^- $ and $ \lambda_n^+,$
$\; n \in \mathcal{M}^\pm.$

(a) If $\Delta \subset \Gamma^\pm, $ then the system $\{u_{2n-1},
\,u_{2n}, \; n\in \Delta \cap \mathcal{M}^\pm\} $ is a (Riesz) basis
in its closed linear span if and only if
\begin{equation}
\label{cr11} \limsup_{n\in \Delta \cap \mathcal{M}^\pm} t_n (z_n^*) <
\infty,
\end{equation}
where $z_n^* = \frac{1}{2} (\lambda^-_n + \lambda^+_n) -\lambda^0_n $
with $\lambda^0_n = n^2 $ for Hill operators and $\lambda^0_n = n $
for Dirac operators.

(b) The system of root functions of $L$ contains a Riesz basis if and
only if (\ref{cr11}) holds for $\Delta =\Gamma^\pm.$
\end{Criterion}

Another interesting abstract criterion of basisness is the following.

\begin{Criterion}
\label{crit3} The system of root functions of the operator
$L_{Per^\pm} (v)$ contains a Riesz basis in $L^2([0,\pi])$ if only
if
\begin{equation}
\label{a6} \limsup_{n \in \mathcal{M}^\pm} \frac{|\lambda_n^+-
\mu_n|}{|\lambda_n^+ - \lambda_n^-|} <\infty.
\end{equation}
\end{Criterion}

This criterion was given (with completely different proofs) in
\cite{GT12} for Hill operators with $L^2$-potentials and in
\cite{DM28} for Hill operators with $H^{-1}_{per}$-potentials and for
one-dimensional Dirac operators with $L^2$-potentials as well.

Recently we have obtained in \cite{DM31}  asymptotic formulas for
spectral gaps $\gamma_n $ and deviations $\delta_n = \mu_n -
\lambda_n^+ $  under the assumptions \eqref{a1} and \eqref{a2}. The
following holds.
\begin{Proposition}
\label{thm51} Assume that there is an infinite set $\Delta \subset \mathbb{N}$
such that
(\ref{a1}) and (\ref{a2}) hold.
Then there exist
branches  $\sqrt{\beta_n^- (z)}$  and
$\sqrt{\beta_n^+ (z)}$  such that
\begin{equation}
\label{50.3} \gamma_n \sim 2\sqrt{\beta_n^- (z_n^*)} \sqrt{\beta_n^+
(z_n^*)},  \quad n\in \Delta
\end{equation}
Moreover

(a) If $\;-1$ is not a cluster point of the sequence $\left (
\sqrt{\beta^-_n (z^*_n)} / \sqrt{\beta^+_n (z^*_n)}\right)_{n\in
\Delta},$ then
\begin{equation}
\label{50.4} \mu_n - \lambda_n^+ \sim  -\frac{1}{2} \left
(\sqrt{\beta^+_n (z_n^*)} + \sqrt{\beta^-_n (z_n^*)} \right )^2,
\quad n \in \Delta,
\end{equation}

(b) If $\,1$ is not a cluster point of the sequence $\left (
\sqrt{\beta^-_n (z^*_n)}/\sqrt{\beta^+_n (z^*_n)}\right)_{n\in
\Delta},$ then
\begin{equation}
\label{50.8} \mu_n - \lambda_n^- \sim  -\frac{1}{2} \left
(\sqrt{\beta^+_n (z_n^*)} - \sqrt{\beta^-_n (z_n^*)} \right )^2,
\quad  n \in \Delta.
\end{equation}

(c) If $-1$ is not a cluster point of the sequence $\left (\beta^-_n
(z^*_n)/\beta^+_n (z^*_n)\right )_{n \in \Delta},$ then in the Hill
case
\begin{equation}
\label{500.4} \mu_n - \frac{1}{2} \left (\lambda_n^- + \lambda_n^+
\right ) \sim -\frac{1}{2} \left ( \beta^+_n (z_n^*) + \beta^-_n
(z_n^*)\right ), \quad n \in \Delta,
\end{equation}
\end{Proposition}
Here and thereafter, we write for two sequences $(a_n)$
and $(b_n)$ that $a_n \sim b_n $  as $n \to \infty $
if $a_n/b_n \to 1 $ as $n \to \infty. $
We write  $a_n \asymp b_n $   if there are constants $C>c>0 $
such that $c \,a_n  \leq b_n \leq C \, a_n $ for large enough $n.$

In this paper paper we study the class of Hill potentials $v$ with
the property that the main term in the asymptotics of $\beta_n^\pm $
equals the Fourier coefficient $V(\pm 2n).$ In the context of Sobolev
spaces, a natural example of such potentials is given by the
following assertion (compare to (\ref{m1}), (\ref{m2}); see also
\cite{VS09}).

\begin{Lemma}
\label{lem00}
Suppose $v(x), \;  0 \leq x \leq \pi,$  is $m $ times differentiable
and the function $v^{(m)}(x) \,$  is
absolutely continuous.  If the conditions

(a)   $v^{(s)} (\pi) = v^{(s)} (0)$  for $s=0, \ldots, m-1$  (if $m>0$)

(b)   $v^{(m)} (\pi) \neq  v^{(m)} (0)$

hold, then we have
\begin{align}
\label{9.3}
\beta_n^- (z) & \sim V(-2n) \sim
\frac{1}{(-2in)^{m+1}} \left ( v^{(m)}(0) - v^{(m)}(\pi)   \right ) ,
\quad |z| \leq n,
 \\ \label{9.4}
\beta_n^+ (z) &  \sim V(2n) \sim
\frac{1}{(2in)^{m+1}} \left ( v^{(m)}(0) - v^{(m)}(\pi)   \right ),
\quad |z| \leq n.
\end{align}

\end{Lemma}

In section 3 we introduce weighted spaces of Hill potentials (in
terms of their Fourier coefficients) and consider general classes of
potentials such that $\beta_n^\pm \sim V(\pm 2n)$  -- see
Theorem~\ref{thm1}. Lemma~\ref{lem00} is a partial case of that
theorem, which corresponds to the weight $\Omega (k) = k^m.$

\section{Weights and weighted spaces}

Since we study the Hill operator on $[0,\pi],$ our basic index set is
$2\mathbb{Z}.$ A sequence of {\em positive} numbers $\Omega = (\Omega
(k))_{k \in 2\mathbb{Z}} $ is called {\em weight}, or {\em weight
sequence}. We consider only even weights, i.e.,
\begin{equation}
\label{31.1} \Omega (-k) = \Omega (k),  \quad k \in 2\mathbb{Z},
\end{equation}
such that
\begin{equation}
\label{31.1a} \Omega (0) = 1, \quad \Omega (k) \leq \Omega (m) \quad
\text{for}  \;\; m \geq k \geq 0.
\end{equation}

For every weight $\Omega $ we consider the corresponding
$\ell^\infty$-type weighted space of Hill potentials
\begin{equation}
\label{31.1b}
W_{\infty} (\Omega ) = \left \{v(x) = \sum_{k \in 2\mathbb{Z}} V(k)
e^{ikx} \;: \; \|v\|_\Omega =\sup_{k\in 2\mathbb{Z}} |V(k)| \Omega (k) <\infty  \right \}.
\end{equation}

We say that two weights $\Omega_1 $ and $\Omega_2 $ are {\em
equivalent} if
\begin{equation}
\label{31.2} \exists\, C \geq 1 \; : \quad C^{-1} \Omega_1 (k)
\leq \Omega_2 (k)  \leq C \Omega_1 (k), \quad    k\in 2\mathbb{Z}.
\end{equation}
Obviously, equivalent weights
generate one and the same weighted space.

A weight $\Omega $ is called {\em submultiplicative} if
\begin{equation}
\label{31.3} \Omega (k+m) \leq \Omega (k) \Omega (m), \quad k,m
\in 2\mathbb{Z}.
\end{equation}

Of course, if $\Omega_1 $ and $\Omega_2 $ are equivalent weights
and one of them is submultiplicative, then the other one
satisfies
\begin{equation}
\label{31.3a} \Omega (k+m) \leq C \Omega (k) \Omega (m), \quad k,m
\in 2\mathbb{Z}
\end{equation}
for some constant $C>0.$ Obviously, if $\Omega $ satisfies
(\ref{31.3a}), then $\tilde{\Omega} = C \Omega $ satisfies
(\ref{31.3}). Moreover, it is easy to see that if (\ref{31.3a})
holds for $ |k|, |m| \geq k_0, $ then it holds for all $ k,m \in
2\mathbb{Z},$ maybe with another constant $C.$
In the sequel we call a weight {\em almost submultiplicative}
if it satisfies (\ref{31.3a}).

A weight $\omega $ is called {\em slowly increasing}
if
\begin{equation}
\label{31.4}  A:=
\sup_{k\in 2\mathbb{N}}   \omega (2k)/ \omega (k) < \infty .
\end{equation}
Every slowly increasing weight is almost submultiplicative. Indeed,
if $0 < k \leq m $  then from (\ref{31.4}) and (\ref{31.1a}) it
follows that
\begin{equation}
\label{31.5} \omega (m+k) \leq \omega (2m) \leq  A \omega (m) \leq A
\omega (m) \omega (k),
\end{equation}
so (\ref{31.3a}) holds with $C=A. $

If $\sup   \Omega (2k)/ \Omega (k) = \infty $ (i.e., if $\Omega
$ is not slowly increasing), then $\Omega $ is called {\em rapidly
increasing} weight. A rapidly increasing submultiplicative
weight $\Omega $ is growing at most exponentially because $$
\Omega (k) \leq (\Omega (2))^{k/2} = e^{a k}, \quad a = \frac{1}{2} \log
\Omega (2).$$

Each weight may be written in the form
\begin{equation}
\label{31.6} \Omega (k) = \exp (h (|k|)), \quad \text{where}\quad
h(k) = \log \Omega (k), \quad h(0) =0.
\end{equation}
Then $\Omega $ is
submultiplicative if and only if $h$ is subadditive, i.e.,
\begin{equation}
\label{31.7}  h(k+m) \leq h(k) + h(m)   \quad  \forall \,k,m \in
2\mathbb{N}.
\end{equation}

It is well known (e.g., see  \cite[Problem 98]{PS}) that if $(h(k)) $ is a
subadditive sequence, then the limit
\begin{equation}
\label{31.8} \ell = \lim_{k\to \infty} \frac{h(k)}{k}
\end{equation}
exists. A submultiplicative weight $\Omega $ of the form
(\ref{31.6}) is called {\em subexponential} if $\ell=0,$ and
{\em exponential} if $\ell >0. $

\begin{Lemma}
\label{lem1}  Let $\Omega $ be a weight of the form (\ref{31.6}).
If the corresponding sequence  $(h(k))_{k\in 2\mathbb{Z}_+}$ is
concave, i.e.,
\begin{equation}
\label{31.10}   h(k+4) - h(k+2) \leq h(k+2) - h(k) \quad\text{for}
\;\; k \geq 0,
\end{equation}
then $\Omega $ is submultiplicative.
\end{Lemma}

\begin{proof}
Fix $k,m \in \mathbb{N}.$ By (\ref{31.10}), we have
$$
\begin{aligned}
h(2k+2m) - h (2k) &= \sum_{i=1}^m [h(2k+2j)-h(2k+2j-2)] \\
&\leq \sum_{j=1}^m
[h(2j)-h(2j-2)] = h(2m).
\end{aligned}
$$

Thus (\ref{31.7}) holds, i.e., the weight
$\Omega (k) = \exp (h(|k|)) $ is submultiplicative.
\end{proof}

Typical examples of submultiplicative weights are
\begin{equation}
\label{31.11} \omega_a (0) = 1, \quad \omega_a (k) = |k|^a  \quad
\text{for} \;\; k \neq 0, \; a > 0
\end{equation}
(known as the Sobolev weights), and
\begin{equation}
\label{31.12} \Omega_{c,\gamma} (k) = \exp (c|k|^\gamma ) , \quad c>
0, \; \gamma \in (0,1)
\end{equation}
(known as the Gevrey weights). The corresponding functions $h$ are
concave.

Further we need the following technical assertion.
\begin{Lemma}
\label{lem2}
For every $c>0, \; \gamma \in (0,1) $ and $a>0 $
the weight
$\Omega = (\Omega (k))_{k \in 2\mathbb{Z}} $  defined by
\begin{equation}
\label{31.13}
\Omega (0)= 1, \quad
\Omega (k) = \exp (c|k|^\gamma ) \, |k|^{-a} \quad  \text{for} \;\;
k \neq 0,
\end{equation}
is almost submultiplicative. Moreover, if
$a\leq c\gamma (1-\gamma)2^\gamma,$ then the weight $\Omega $
is submultiplicative.

\end{Lemma}

\begin{proof}
We have  $\Omega (k) = e^{h(|k|)},$ where
$$
h(0)= 0, \quad h(x) = c \, x^\gamma - a \log x \quad \text{for}
\;\; x>0.
$$
For large enough $x$ the function $h$ is concave. Indeed,
$$
h^{\prime \prime} (x) =  \frac{1}{x^2}
\left ( c \gamma (\gamma -1)x^\gamma + a \right )
<0 \quad \text{for} \;\; x>x_0,
$$
where $x_0 = \left [ \frac{a}{c\gamma (1-\gamma)} \right ]^{1/\gamma} . $
Set
$$
h_1 (x) = \begin{cases}   d+ h(x)  &  \text{for}\;\;  x >x_0, \\
[d+h(x_0)] \frac{x}{x_0}  &    \text{for}  \;\; 0\leq x \leq x_0,
\end{cases}
$$
where the constant $d>0 $  is chosen so large that
$$[d+h(x_0)]/x_0 \geq h^\prime (x_0) \quad \text{and} \quad
h(x) \leq   [d+h(x_0)]\frac{x}{x_0}  \quad \text{for} \;\; 2\leq x \leq x_0.
$$
Then $h_1  $  is a concave function on $[0, \infty) $ with $h_1 (0) = 0,$ so
by Lemma~\ref{lem1}   the weight $\Omega_1 (k) = e^{h_1 (|k|)} $
is submultiplicative.  Since
$$
h(x) \leq h_1 (x) \leq h(x) + d + h(x_0) \quad \text{for} \;\; x \geq 2,
$$
the weights $\Omega $ and $\Omega_1 $ are equivalent.

If $a\leq c\gamma (1-\gamma)2^\gamma,$
then $x_0 = \left [ \frac{a}{c\gamma (1-\gamma)}
\right ]^{1/\gamma}  \leq 2, $
so it follows that the function $h$ is concave for $x\geq 2. $
Thus,  (\ref{31.10}) holds for $k \geq 2. $

If $k=0, $ then (\ref{31.10})  reduces to $h(4) \leq 2 h(2), $  i.e.,
$$
c \, 4^\gamma - a \log 4 \leq 2 (c \, 2^\gamma - a \log 2) =
c\, 2^{1+\gamma} - a \log 4.
$$
Since $4^\gamma = 2^{2\gamma} < 2^{1+\gamma}, $
(\ref{31.10}) holds for $k=0$ as well, so by Lemma~\ref{lem1}
it follows that in this case the weight $\Omega $ is submultiplicative.

\end{proof}

\section{Main results}

\begin{Theorem}
\label{thm1}
Suppose
$\Omega = (\Omega (k))_{k \in 2\mathbb{Z}}$
is a weight of the form
\begin{equation}
\label{31.20} \Omega (k) =\omega (k) \cdot
 \tilde{\Omega} (k) ,
\end{equation}
where $\tilde{\Omega}$ is an almost submultiplicative weight
and
$\omega $  is a slowly increasing weight with
\begin{equation}
\label{31.21}
M: = \sum_{k\neq 0} \frac{1}{|k| \,\omega (k)} < \infty.
\end{equation}
Let $v\in W_{\infty} (\Omega), $  and let
 $(V(k))_{k \in 2\mathbb{Z}} $  be its Fourier coefficients.
\bigskip

(a) If $\Delta \subset \mathbb{N}$ is an infinite set such that
\begin{equation}
\label{31.22} |V(\pm 2n)| \, n\,  \Omega (2n) \to \infty  \quad
\text{as} \;\;  n \in \Delta, \;\; n \to \infty,
\end{equation}
then
\begin{equation}
\label{31.23} \beta_n^\pm (v, z) \sim V(\pm 2n) \quad \text{as} \;\;
|z|\leq n/2, \; n \in \Delta, \; n \to \infty.
\end{equation}
\bigskip

(b)  If
 \begin{equation}
\label{31.210}
 \lim_{|k|\to \infty} |V(k)|\Omega (k) =  0
\end{equation}
and $\Delta \subset \mathbb{N}$ is an infinite set such that
\begin{equation}
\label{31.220}    \exists c>0: \;\; |V(\pm 2n)| \, n\, \Omega (2n)
\geq c \quad \text{for} \quad n \in \Delta,
\end{equation}
then (\ref{31.23}) holds.

\end{Theorem}

\begin{proof}
We prove (\ref{31.23}) for $\beta_n^+ $ only;  the proof for
$\beta_n^-$  is the same.

In view of (\ref{22.9}),
\begin{equation}
\label{31.25} |\beta_n^+ (z) - V(2n)| \leq \sum_{k=1}^\infty
|S_k^{21}(z)|,
\end{equation}
where $S_k^{21}$ are defined by \eqref{22.15}.

Set
\begin{equation}
\label{31.26}
r(k) = |V(k)| \, \Omega (k), \quad R_m = \sup_{|k|\geq m}  r(k).
\end{equation}
By $v \in W_{\infty}(\Omega), $  we have that
$r(k)  \leq  \|v\|_\Omega, $
so $R_m \leq \|v\|_\Omega.$
 Moreover,  since $\tilde{\Omega}$
 satisfies (\ref{31.3a}) with some constant $C$ and
$\omega $ satisfies (\ref{31.5}) with a constant $A,$
 it follows that
\begin{equation}
\label{31.27}
\begin{aligned}
& |V(n-j_1)V(j_1 - j_2) \cdots V(j_{k-1} - j_k) V(j_k +n)| \, \Omega
(2n) \\ &\leq   (AC)^k r(n-j_1)r(j_1 - j_2) \cdots r(j_{k-1} - j_k)
r(j_k +n).
\end{aligned}
\end{equation}

First we estimate $S^{21}_1. $   By (\ref{22.15}),
(\ref{31.20}) and (\ref{31.26}),
$$
|S_1^{21}(z)| \, \Omega (2n) \leq \sum_{j \in n +
2\mathbb{Z}\setminus \{\pm n\} } \frac{C r(n-j)r(j+n)}{|n^2 - j^2
+z|}  \cdot \frac{\omega (2n)}{\omega (n-j)\omega (j+n)} .
$$
It is easy to see that
\begin{equation} \label{31.28} |n^2 -j^2 +z|
\geq |n^2 -j^2 | -|z| \geq \frac{1}{2}|n^2 -j^2 | \quad \text{if}
\;\; |z| \leq n/2.
\end{equation}
Therefore, for $|z|\leq n/2 $ we have
$$
|S_1^{21}(z)| \, \Omega (2n) \leq   \sigma_1 + \sigma_2,
$$
with
$$
\begin{aligned}
\sigma_1 &=
\sum_{j <0, \, j\neq -n}
\frac{2Cr(n-j)r(j+n)}{|n-j|\,|n+j|}  \cdot
\frac{\omega (2n)}{\omega (n-j)\omega (j+n)}\\
&\leq   \sum_{j <0, \, j\neq -n}
\frac{2CR_n \|v\|_{\Omega} }{n|n+j| \,\omega (n+j)}
\cdot \frac{\omega (2n)}{\omega (n)}
\quad (\text{by (\ref{31.26}) and} \;\;  n-j \geq n)\\
& \leq 2AC\|v\|_{\Omega} R_n \, \frac{1}{n}  \sum_{j\neq -n}
\frac{1}{|n+j | \,
\omega (n+j)}
\leq  2ACM\|v\|_{\Omega} R_n \, \frac{1}{n}
\end{aligned}
$$
by (\ref{31.4}) and (\ref{31.21}), and similarly,
$$
\begin{aligned}
\sigma_2 &=
\sum_{j >0, \, j\neq n}
\frac{2Cr(n-j)r(j+n)}{|n-j|\,|n+j|}  \cdot
\frac{\omega (2n)}{\omega (n-j)\omega (j+n)}\\
&\leq  2ACM\|v\|_{\Omega} R_n  \,\frac{1}{n}.
\end{aligned}
$$
Since $R_n \leq \|v\|_\Omega,$ it follows that
\begin{equation}
\label{31.29} |S_1^{21}(z) |  = O \left (  \frac{1}{n \, \Omega (2n)}
\right ), \quad  \;\; |z|\leq n/2.
\end{equation}

If (\ref{31.210}) holds, then $R_n \to 0, $ so we obtain that
\begin{equation}
\label{31.30} |S_1^{21}(z) | = o \left (\frac{1}{n \, \Omega (2n)}
\right), \quad  \;\; |z|\leq n/2.
\end{equation}

Next we estimate $S_k^{21}$  for $k=2,3, \ldots.$ In view of
(\ref{22.15}), (\ref{31.20}) and  (\ref{31.26})--(\ref{31.28}), we
have
$$
\begin{aligned}
|S_k^{21}(z)| \, \Omega (2n)  & \leq \sum_{j_1, \ldots j_k \neq \pm
n} 2^k (AC)^k \frac{r(n-j_1) r(j_1 -j_2) \cdots r(j_k +n)}
{|n^2 - j_1^2| \, |n^2 - j_2^2| \cdots |n^2 - j_k^2|}\\
& \leq \|v\|_{\Omega}^{k+1}  (2AC)^k \left ( \sum_{j\neq \pm n}
\frac{1}{|n^2 - j^2|}  \right )^k.
\end{aligned}
$$
Since $\sum_{j\neq \pm n}  \frac{1}{|n^2 - j^2|} \leq \frac{2\log
(6n)}{n}, $ it follows that
\begin{equation}
\label{31.40} |S_k^{21}(z)| \, \Omega (2n) \leq \|v\|_{\Omega}^{k+1}
(2AC)^k \left (  \frac{2\log (6n)}{n} \right )^k, \quad |z|\leq n/2.
\end{equation}

Now,  if $n$ is so large that
$$\frac{4\|v\|_{\Omega} AC \log (6n)}{n}< \frac{1}{2}, $$
we obtain that
\begin{equation}
\label{31.41} \sum_{k=2}^\infty |S_k^{21}| \leq \|v\|_{\Omega} \left
( \frac{4\|v\|_\Omega \log (6n)}{n}  \right )^2 \frac{1}{\Omega (2n)}
= O \left (\frac{(\log n)^2}{n^2 \Omega (2n)} \right).
\end{equation}
Thus,  if (\ref{31.22}) holds, then (\ref{31.25}), (\ref{31.29}) and
(\ref{31.41}) imply (\ref{31.23}).

Moreover, in the case when (\ref{31.210}) holds, (\ref{31.25}),
(\ref{31.220}), (\ref{31.30})  and (\ref{31.41}) prove (\ref{31.23})
for $\beta_n^+.$
\end{proof}

\begin{Corollary}
Lemma \ref{lem00} holds.
\end{Corollary}

\begin{proof}
Indeed, integration by parts and the Riemann-Lebesgue Lemma show
that
$$
V(k)\sim \frac{1/\pi}{(ik)^{m+1}} (v^{(m)}(0) -v^{(m)}(\pi)).
$$
Consider the weight $\Omega$  defined by
$$
\Omega (0) =1, \quad \Omega (k) = |k|^{m+1} \quad \text{for} \;\;
k\neq 0.
$$
Then $v \in W_\infty (\Omega) $ and $|V(\pm 2n)| \, n \,\Omega
(2n)\to \infty. $

We can apply Theorem~\ref{thm1}, since the weight $\Omega $
satisfies (\ref{31.20}) with $\omega (k) = |k|, \; \tilde{\Omega}(k)
= |k|^m.$ Hence, Lemma~\ref{lem00} follows from (\ref{31.23}).

\end{proof}

In view of Lemma~\ref{lem2} one can apply Theorem~\ref{thm1} to
weights $\Omega $ of the form
\begin{equation}
\label{31.51} \Omega (k) = |k|^a e^{c|k|^\gamma} \quad \text{for}
\;\; k\neq 0, \quad a \in \mathbb{R}, \; c>0, \; \gamma \in (0,1).
\end{equation}
But it is impossible to apply Theorem~\ref{thm1} if the weight
$\Omega$ is growing so slowly that $\sum \frac{1}{|k|\, \Omega (k)}
= \infty.$ For example, this is the case if we consider the weight
$\Omega$ given by
\begin{equation}
\label{31.52} \Omega (0)= 1, \quad  \Omega (k) = \log (e|k|)  \quad
\text{for} \;\; k\neq 0.
\end{equation}
For such weights, the next theorem gives conditions  which guarantee
that $V(\pm 2n) $  is the main term in the asymptotics of
$\beta_n^\pm. $

\begin{Theorem}
\label{thm2}
Suppose
$\Omega = (\Omega (k))_{k \in 2\mathbb{Z}}$
is an almost submultiplicative weight.
Let $v\in W_{\infty} (\Omega) $,  and let
 $(V(k))_{k \in 2\mathbb{Z}} $  be its Fourier coefficients.
\bigskip

(a) If $\Delta \subset \mathbb{N}$ is an infinite set such that
\begin{equation}
\label{41.22}
|V(\pm 2n)| \, \frac{n}{\log n}\,
 \Omega (2n) \to \infty  \quad \text{as} \;\;
   n \in \Delta,  \;\; n \to \infty,
\end{equation}
then
\begin{equation}
\label{41.23} \beta_n^\pm (v, z) \sim V(\pm 2n) \quad \text{as} \;\;
|z| \leq n/2, \: n \in \Delta, \; n \to \infty.
\end{equation}
\bigskip

(b)  If
 \begin{equation}
\label{41.210}
 \lim_{|k|\to \infty} |V(k)|\Omega (k) = 0
\end{equation}
and $\Delta \subset \mathbb{N}$ is an infinite set such that
\begin{equation}
\label{41.220} \exists c>0: \quad |V(\pm 2n)| \, \frac{n}{\log n}\,
\Omega (2n) \geq c \quad \text{for} \quad n \in \Delta,
\end{equation}
then (\ref{41.23}) holds.

\end{Theorem}

\begin{proof}
As in the proof of Theorem~\ref{thm1}, we consider $\beta_n^+$ only
and use the notations (\ref{31.25}) and (\ref{31.26}).

Since the weight $\Omega $ is almost submultiplicative,  we have
\begin{equation}
\label{41.27}
\begin{aligned}
& |V(n-j_1)V(j_1 - j_2) \cdots V(j_{k-1} - j_k) V(j_k +n)| \, \Omega
(2n) \\ &\leq   C^k r(n-j_1)r(j_1 - j_2) \cdots r(j_{k-1} - j_k)
r(j_k +n)
\end{aligned}
\end{equation}
with some constant $C\geq 1. $

As in the proof of Theorem~\ref{thm1} we obtain
\begin{equation}
\label{41.410} \sum_{k=2}^\infty |S_k^{21}(z)| = O \left (\frac{(\log
n)^2}{n^2 \Omega (2n)} \right), \quad |z|\leq n/2.
\end{equation}

Next we estimate $S^{21}_1. $   By (\ref{22.15}), (\ref{31.26}) and
(\ref{41.27}) we obtain (since $r(n-j)r(n+j) \leq R_n \|v\|_\Omega$)
$$
\begin{aligned}
|S_1^{21}(z)| \, \Omega (2n) &\leq \sum_{j \in (n +
2\mathbb{Z})\setminus \{\pm n\} } \frac{2C r(n-j)r(j+n)}{|n^2 -
j^2|}\\
& \leq 2C R_n \|v\|_\Omega \sum_{j\neq \pm n} \frac{1}{|n^2 - j^2|}
\\&\leq  4C R_n \|v\|_\Omega  \frac{\log (6n)}{n}.
\end{aligned}
$$
Since $R_n \leq \|v\|_\Omega,$ it follows that
\begin{equation}
\label{41.29} |S_1^{21} (z) | = O \left (\frac{\log n}{n \, \Omega
(2n)} \right), \quad |z|\leq n/2.
\end{equation}

If (\ref{41.210}) holds, then $R_n \to 0, $ so we obtain that
\begin{equation}
\label{41.300} |S_1^{21}(z) | = o \left (\frac{\log n}{n \, \Omega
(2n)} \right), \quad |z|\leq n/2.
\end{equation}

Finally,   if (\ref{41.22}) holds, then (\ref{41.410}) and
(\ref{41.29}) imply (\ref{41.23}) for $\beta_n^+.$  Moreover, if
(\ref{41.210}) holds, then (\ref{41.220}), (\ref{41.410}) and
(\ref{41.300}) prove (\ref{41.23}) for $\beta_n^+.$

\end{proof}

\begin{Theorem}
\label{thm4} Let $L=L(v)$ be the Hill operator with a potential $v$
that satisfies (with some infinite set of indices $\Delta \subset
2\mathbb{N}$ or $\Delta \subset 2\mathbb{N}+1$) the assumptions of
either part (a) or part (b) of Theorem~ \ref{thm1}, or either part
(a) or part (b) of Theorem~ \ref{thm2}. Then there are square roots
$\sqrt{V(-2n)}$  and $\sqrt{V(2n)}$ such that:
\begin{equation}
\label{61.3}  (a) \quad \quad \lambda_n^+ -\lambda_n^- \sim  2
\sqrt{V(-2n)} \sqrt{V(2n) } \quad \text{as} \quad n\in \Delta, \;
n\to \infty.
\end{equation}

(b) If $\;-1$ is not a cluster point of the sequence $\left (
\sqrt{V(-2n)} / \sqrt{V(2n)}\right)_{n\in \Delta},$ then
\begin{equation}
\label{61.4} \mu_n - \lambda_n^+ \sim  -\frac{1}{2} \left
(\sqrt{V(-2n)} + \sqrt{V(2n)} \right )^2 \quad \text{as} \; \; n\in
\Delta, \; n\to \infty.
\end{equation}

(c) If $\;1$ is not a cluster point of the sequence $\left (
\sqrt{V(-2n)} / \sqrt{V(2n)}\right)_{n\in \Delta},$ then
\begin{equation}
\label{61.5} \mu_n - \lambda_n^- \sim  -\frac{1}{2} \left
(\sqrt{V(-2n)} - \sqrt{V(2n)} \right )^2 \quad \text{as} \;\; n\in
\Delta, \; n\to \infty.
\end{equation}

(d) If $-1$  is not a cluster point of $(V(-2n)/V(2n))_{n\in \Delta},$ then
\begin{equation}
\label{61.7} \mu_n - \frac{1}{2}(\lambda_n^+ + \lambda_n^-) \sim
-\frac{1}{2} (V(-2n) + V(2n)) \quad \text{as} \;\; n\in \Delta, \;
n\to \infty.
\end{equation}

(e)  Moreover, if $u_n^-, u_n^+ $
are normalized eigenvectors corresponding to
the eigenvalues $\lambda_n^-, \lambda_n^+ ,$  then the system
$\{u_n^\pm, \; n \in \Delta\}$
is a Riesz basis in its closed linear span if and only if
\begin{equation}
\label{61.10} V(-2n) \asymp V(2n) \quad \text{as} \quad  n\in \Delta,
\; n\to \infty.
\end{equation}

\end{Theorem}

Notice that part (e) of Theorem~\ref{thm4} generalizes the results
of Makin \cite{Ma06-1} to a much wider classes of potentials that
include both singular potentials and potentials in Carlemann-Gevrey
classes far beyond the Sobolev spaces.

\section{Examples}

1. Further we say that a Hill operator $L(v)$ has the periodic (or
antiperiodic) Riesz basis property (RBP) if the periodic (or
antiperiodic) root function system of $L(v)$ contains Riesz bases. In
view of Theorems~\ref{thm1}, \ref{thm2} and \ref{thm4} it is easy to
give nontrivial examples of potentials $v$ such that the operator
$L(v)$ has or has not the periodic and/or antiperiodic Riesz basis
property.

Indeed, the relation $v\in W_\infty (\Omega)$ means that the Fourier
coefficients $(V(k))_{k\in 2\mathbb{Z}}$ of a potential $v$ have the
form
\begin{equation}
\label{5.01} V(k) = \frac{\eta (k)}{\Omega (k)} \quad \text{with}
\;\; (\eta (k)) \in \ell^\infty (2\mathbb{Z}).
\end{equation}
Conversely, we may determine a potential $v$ by defining its Fourier
coefficients by (\ref{5.01}).

Theorems~\ref{thm1} and \ref{thm4} imply immediately the following.
\begin{Proposition}
Let $\Omega$ be a weight that satisfies the conditions \eqref{31.20}
and \eqref{31.21}.  Choose a bounded scalar sequence $(\eta
(k))_{k\in 2\mathbb{Z}} $ so that
$$
\begin{aligned}
&(i) \quad n \cdot \eta (\pm 2n) \to \infty \quad \text{as} \quad n
\in
\mathbb{N},\;\; n \to \infty;\\
&(ii) \quad  \eta ( -2n) \asymp \eta (2n) \quad  \text{or} \quad
(ii^*) \quad  \eta (-2n) \not \asymp \eta (2n) \quad  \text{as} \quad
n \in 2\mathbb{N};\\
&(iii) \;\; \eta ( -2n) \asymp \eta (2n) \quad  \text{or} \quad
(iii^*) \;\;  \eta ( -2n) \not \asymp \eta (2n) \quad  \text{as}
\quad n \in 1+ 2\mathbb{N}.
\end{aligned}
$$
Then the operator $L(v)$ with a potential $v$ given by (\ref{5.01})
has/(has not) the periodic RBP if respectively (ii)/$(ii^*)$ holds,
and has/(has not) the antiperiodic RBP if respectively
(iii)/$(iii^*)$ holds.
\end{Proposition}

To be more specific, let us consider the following example, where
conditions (ii) and $(iii^*)$ hold.
\begin{Example}
\label{example1} Let $v$  be defined by (\ref{5.01}) with a sequence
$(\eta (k)) $ given by $\eta (0) =0$ and
\begin{equation}
\label{5.02} \eta (2n) = \frac{\log n}{n} \quad \text{for} \;\; n \in
\mathbb{N}, \quad \eta (-2n)= \begin{cases} (\log n)/n &
 \text{for} \;\; n \in 2\mathbb{N},\\
(\log n)^2/n & \text{for} \;\; n \in 1+ 2\mathbb{N}.
\end{cases}
\end{equation}
Then the operator $L(v) $ has the periodic Riesz basis property and
fails the antiperiodic Riesz basis property.
\end{Example}

Notice that we cannot apply Theorem~\ref{thm2} to the case given by
Example~\ref{example1}. But on the other hand Theorem~\ref{thm2}
works for a wider class of potentials as the following example shows.
\begin{Example}
\label{example2} Consider the potential
$$
\Omega = (\Omega (k)), \quad \Omega (k)=1 \;\; \forall k \in
2\mathbb{Z}.
$$
Let $v$ be the potential defined formally by its Fourier
coefficients $(V(k))$ given by $V(0) =0$ and
\begin{equation}
\label{5.03} V (2n) = 1 \quad \text{for} \;\; n \in \mathbb{N}, \quad
V (-2n)= \begin{cases} 1/\sqrt{n} &
 \text{for} \;\; n \in 2\mathbb{N},\\
1 & \text{for} \;\; n \in 1+ 2\mathbb{N}.
\end{cases}
\end{equation}
Then, by Theorems \ref{thm2} and \ref{thm4}, the operator $L(v) $ has
the antiperiodic Riesz basis property and fails the periodic Riesz
basis property.
\end{Example}
Of course, one can easily modify the above example and get a
potential $v$ such that $L(v) $ has (or fails) the periodic RBP and
fails the antiperiodic RBP.

\bigskip

2. In Section 2 we consider classes of potentials $v$ such that
$\beta_n^\pm (z) \sim V(\pm 2n), $ where $(V(k)) $ are the Fourier
coefficients of $v.$ By (\ref{22.9}),
\begin{equation}
\label{4.1} \beta_n^- (z) = V(-2n) + \sum_{k=1}^\infty  S^{12}_k
(n,z), \quad \beta_n^+ (z) =V(2n) + \sum_{k=1}^\infty  S^{21}_k
(n,z),
\end{equation}
where $S^{12}_k$  and $S^{21}_k $ are given by
(\ref{22.14}) and (\ref{22.15}).
Of course, for a generic potential $v$  it is not true that
the first term $V(\pm 2n) $ of the series defining  $\beta_n^\pm (z)$
dominates the sum of all others and determines the asymptotics.

Moreover,  let $v$ be  a trigonometric polynomial, say
$$v(x) =
\sum_{|k| \leq M}   V(k) e^{ikx}. $$ Every term of the sum $S^{21}_k$
(or $S^{12}_k$) given by (\ref{22.14}) or (\ref{22.15}) is a fraction
which numerator has the form
$$
V(\pm n-j_1) V(j_1 - j_2) \cdots V(j_{k-1} - j_k) V(j_k \pm n).
$$
Notice, that
$$(\pm n-j_1) + (j_1 -j_2) + \cdots + (j_{k-1} - j_k) +(j_k \pm
n) = \pm 2n. $$  Therefore, if $(k+1)M < 2n $ then the absolute value
of one of these numbers will be strictly greater than $M,$ so the
corresponding Fourier coefficient will be zero. Thus, whenever
$(k+1)M < 2n $ we have $S^{21}_k = 0$ and $S^{12}_k = 0.$   In other
words, if $v$ is a trigonometric polynomial then no fixed partial sum
of the series in (\ref{4.1}) gives the asymptotics of $\beta_n^\pm
(z).$ We refer to \cite{DM25a, DM25, DM27, DM31}  for results about
the asymptotics of $\beta_n^\pm, $ $\gamma_n = \lambda_n^+ -
\lambda_n^-, $ $\delta_n =\mu_n - \lambda_n^+$ and Riesz basis
property of root function systems in the case of potentials that are
trigonometric polynomials. See also \cite{AS, AD1} and the
bibliography therein.

The situation is similar in the case of potentials which Fourier
coefficients (by absolute value) decay superexponentially, i.e.,
$$
\exists \, \gamma >1: \quad |V(k)| \leq  e^{-|k|^\gamma}, \quad
n\geq N_*.
$$
In \cite{DM4}, it is shown that no fixed partial sum of the series
in (\ref{4.1}) gives the asymptotics of $\beta_n^\pm (z).$

In the context of Sobolev spaces $W^m_1$ Shkalikov and Veliev
\cite[Theorems 2-4]{VS09} gave conditions on $v$ for existence (or
nonexistence) of Riesz bases (consisting of periodic or antiperiodic
root functions)  in terms of partial sums
\begin{equation}
\label{4.3} \Sigma^-_m (n,z)= V(-2n) + \sum_{k=1}^m  S^{12}_k(n,z),
\quad \Sigma^+_m (n,z)=V(2n) + \sum_{k=1}^m  S^{21}_k (n,z).
\end{equation}
In fact the assumptions of Theorem 2 in \cite{VS09} say that the
partial sums $\Sigma^-_m $  and $\Sigma^+_m$ give the main terms in
the asymptotics of $\beta_n^- $ and $\beta_n^+ $ respectively.

The following statement generalizes part (a) of Theorem~\ref{thm2}.
Moreover, in view of Criterion~\ref{crit1} it could be considered as
a generalization of the results of Shkalikov and Veliev \cite{VS09}.

\begin{Theorem}
\label{thm20} Suppose $\Omega = (\Omega (k))_{k \in 2\mathbb{Z}}$ is
an almost submultiplicative weight.  Let $v\in W_{\infty} (\Omega)
$, and let
 $(V(k))_{k \in 2\mathbb{Z}} $  be the Fourier coefficients of $v.$

(a) If $\Delta \subset \mathbb{N}$ is an infinite set such that
\begin{equation}
\label{20.22} |\Sigma_m^\pm (n, z_n^*)| \, \Omega (2n) \left
(\frac{n}{\log n}\right )^{m+1}\,
 \to \infty  \quad \text{as} \;\; n \in \Delta,  \; |n| \to \infty,
\end{equation}
then
\begin{equation}
\label{20.23} \beta_n^\pm (v, z_n^*) \sim \Sigma_m^\pm (n,z_n^*)
\quad \text{as}  \; \; n \in \Delta, \; n \to \infty.
\end{equation}

\end{Theorem}

\begin{proof}
As in the proof of Theorem \ref{thm1}, we have
\begin{equation}
\label{20.40} |S_k^{21}(z)| \, \Omega (2n) \leq \|v\|_{\Omega}^{k+1}
(2C)^k \left (  \frac{2\log (6n)}{n} \right )^k, \quad |z|\leq n/2,
\end{equation}
which leads to
\begin{equation}
\label{20.41} \sum_{k=m+1}^\infty |S_k^{21} (z)|  = O \left
(\frac{(\log n)^2}{n^2 \Omega (2n)} \right), \quad |z|\leq n/2.
\end{equation}
Now, (\ref{4.1}),  (\ref{20.41}) and (\ref{20.22}) imply
(\ref{20.23}).

\end{proof}

3. Next we give examples where the asymptotics of $\beta_n^\pm $ is
determined by  $S^{12}_1$  and $S^{21}_1$ but not by $V(\pm 2n).$

\begin{Proposition}
\label{ex1} Let $\Omega = (\Omega (k))_{k \in 2\mathbb{Z}}$ be an
almost submultiplicative weight that satisfies the conditions of
Theorem~\ref{thm1}, and let $v$ be the potential with Fourier
coefficients $(V (k))_{k \in 2\mathbb{Z}} $ defined by
\begin{equation}
\label{4.11}
V(\pm 2) = \pm \frac{1}{\Omega (2)},
\end{equation}
\begin{equation}
\label{4.12}
V(\pm  4p) = \pm \frac{1}{\Omega (4p)},   \quad p \in \mathbb{N},
\end{equation}
\begin{equation}
\label{4.13} V(4p+2) = \frac{\xi_p}{p \,\Omega (4p+2)},    \quad
 \xi_p \geq 0, \quad p \in \mathbb{N},
\end{equation}
\begin{equation}
\label{4.14} V(-4p-2) = - \frac{\eta_p}{p \, \Omega (4p+2)}, \quad
\eta_p \geq 0, \quad p \in \mathbb{N}.
\end{equation}
If
\begin{equation}
\label{4.16}
 \xi_p \to 0 \quad  \text{and} \quad  \eta_p \to 0,
\end{equation}
then
\begin{equation}
\label{4.17} \left |\beta_{2p+1}^\pm (z_{2p+1}^*)\right |  \asymp
\frac{1}{p \, \Omega (4p)},
\end{equation}
and there is a Riesz basis in $L^2 ([0,\pi])$ which consists of
antiperiodic  root functions.
\end{Proposition}

\begin{proof}
In view of Criterion~\ref{crit1}, (\ref{4.17}) implies that the
system of antiperiodic root functions contains Riesz bases.
Therefore, we need to prove (\ref{4.17}) only.

By \eqref{22.9} we have
\begin{equation*}
 \left |\beta^+_{2p+1} (z_{2p+1}^*) - S_1^{21} (2p+1,
z_{2p+1}^*) \right | \leq |V(2n)| + \sum_{k=2}^\infty \left |
S^{21}_k (2p+1,z_{2p+1}^*)\right |.
\end{equation*}
From \eqref{4.12}--\eqref{4.16} it follows that $v \in W_\infty
(\Omega), $ so (\ref{31.41}) holds since its proof uses only that $v
\in W_\infty (\Omega). $ Therefore, in view of  (\ref{4.13}) we
obtain that
\begin{equation}
\label{4.19} \left |\beta^+_{2p+1} (z_{2p+1}^*) - S_1^{21} (2p+1,
z_{2p+1}^*) \right | = o \left ( \frac{1}{p\, \Omega (4p) } \right )
\end{equation}

Next we estimate $S_1^{21} (2p+1, z_{2p+1}^*).  $ Consider $S_1^{21}
(n,0) $ with $n=2p+1. $ It is easy to see that
$$S_1^{21} (n,0) = \sum_{j \neq n}
\frac{V(n-j)V(j +n) }{n^2 -j^2}$$ is a sum of positive terms. Indeed,
if $-n < j< n,$ then $n^2 - j^2 >0 $  and $V(n\pm j) > 0$ due to
(\ref{4.11})--(\ref{4.14}), so the corresponding term is positive. If
$j>n $ or $j< -n,$  then $n^2 - j^2 <0 $ and  either $V(n-j) <0, \,
V(n+j) > 0$  or   $V(n-j) > 0, \, V(n+j) < 0$ so again the
corresponding term is positive.

Therefore, we have
\begin{equation}
\label{4.20}
S_1^{21} (2p+1,0) > \frac{V(4p)V(2)}{8p} =
\frac{1}{8p \,\Omega (4p) \Omega (2)},
\end{equation}
where the expression on the right is the term of $S_1^{21} (n,0)$
associated with $j= 2p-1.$

Next we show that
\begin{equation}
\label{4.22} A:= \left |S_1^{21} (2p+1,z_{2p+1}^*)- S_1^{21} (2p+1,0)
\right | =o \left ( \frac{1}{p \,\Omega (4p)} \right ).
\end{equation}
We have, with $n= 2p+1 $ and $z=z_{2p+1}^*,$
$$
\begin{aligned}
A  &\leq  \sum_{j\neq \pm n}  \left |  \frac{V(n-j)V(j+n)}{n^2-j^2
+z} -
\frac{V(n-j)V(j+n)}{n^2-j^2 }  \right | \\
& =
\sum_{j\neq \pm n}  \left |
\frac{V(n-j)V(j+n) \, z}{(n^2-j^2 +z)(n^2-j^2)}  \right |
\end{aligned}
$$
Since $v \in W_\infty (\Omega) $ and the weight $\Omega $ is almost
submultiplicative, we have
$$
|V(n-j)V(j+n)| \leq \frac{\|v\|^2_\Omega}{\Omega (n-j) \Omega (n+j)}
\leq \frac{C\|v\|^2_\Omega}{\Omega (2n)}.
$$
Therefore, from (\ref{31.28}) and the elementary estimate
$\sum_{j\neq \pm n} \frac{1}{(n^2-j^2)^2} \leq \frac{4}{n^2} $ it
follows that
$$
A \leq  \frac{C\|v\|^2_\Omega|z|}{\Omega (2n)} \cdot \sum_{j\neq \pm
n} \frac{2}{(n^2-j^2)^2} \leq \frac{8C\|v\|^2_\Omega |z|}{n^2 \Omega
(2n)}.
$$
On the other hand, by (\ref{i08}) we have
$$
|z_n^*|/n \to 0  \quad \text{as} \;\; n \to \infty
$$
even in the case  $v\in H^{-1}_{per}.$

So, with $n= 2p+1 $ and $z=z_{2p+1}^*$ it follows that (\ref{4.22})
holds.  Now (\ref{31.29}), (\ref{4.20}) and (\ref{4.22}) imply
\begin{equation}
\label{4.24} \left |S_1^{21} (2p+1,z_{2p+1}^*) \right |\asymp
\frac{1}{p \,\Omega (4p)} .
\end{equation}
Now (\ref{4.19}) and (\ref{4.24}) imply (\ref{4.17}) for $\beta_n^+.$

The proof of  (\ref{4.17}) for $\beta_n^-$ is similar. By
\eqref{22.9} we have
\begin{equation*}
 \left |\beta^-_{2p+1} (z_{2p+1}^*) - S_1^{12} (2p+1,
z_{2p+1}^*) \right | \leq |V(-2n)| + \sum_{k=2}^\infty \left |
S^{12}_k (2p+1,z_{2p+1}^*)\right |.
\end{equation*}
One can use the above argument to prove that
\begin{equation}
\label{4.25} \left | S_1^{12} (2p+1,z_{2p+1}^*) \right |\asymp
\frac{1}{p \,\Omega (4p)} .
\end{equation}
Also,  the same argument that proves (\ref{31.41})
shows that
\begin{equation}
\label{31.410} \sum_{k=2}^\infty |S_k^{12}(z)| = O \left (\frac{(\log
n)^2}{n^2 \Omega (2n)} \right), \quad |z|\leq n/2.
\end{equation}
Now  (\ref{4.14}),  (\ref{4.16}), (\ref{4.25})  and (\ref{31.410})
imply (\ref{4.17}) for  $\beta_n^-.$

\end{proof}

Next we modify the construction in Proposition~\ref{ex1} in order
to give examples of potentials without Riesz basis property.

\begin{Proposition}
\label{ex2} Let $\Omega = (\Omega (k))_{k \in 2\mathbb{Z}}$ be an
almost submultiplicative weight that satisfies the conditions of
Theorem~\ref{thm1}, and let $v$ be the potential with Fourier
coefficients $(V (k))_{k \in 2\mathbb{Z}} $ defined by
\begin{equation}
\label{4.110} V(\pm 2) = \pm \frac{1}{\Omega
(2)},
\end{equation}
\begin{equation}
\label{4.120} V( 4p) =  \frac{1}{ \log (4p) \, \Omega (4p)},   \quad
p \in \mathbb{N},
\end{equation}
\begin{equation}
\label{4.125}
V(- 4p) = - \frac{1}{\Omega (4p)},   \quad p \in \mathbb{N},
\end{equation}
\begin{equation}
\label{4.130} V(4p+2) = \frac{\xi_p}{p \,\log (4p)\,\Omega (4p+2)},
\quad
 \xi_p \geq 0, \quad p \in \mathbb{N},
\end{equation}
\begin{equation}
\label{4.140} V(-4p-2) = - \frac{\eta_p}{p \, \Omega (4p+2)}, \quad
\eta_p \geq 0, \quad p \in \mathbb{N}.
\end{equation}
If
\begin{equation}
\label{4.160}
 \xi_p \to 0 \quad  \text{and} \quad  \eta_p \to 0,
\end{equation}
then
\begin{equation}
\label{4.175} \beta_{2p+1}^- (z_{2p+1}^*)  \asymp  \frac{1}{p  \,
\Omega (4p)}.
\end{equation}
and
\begin{equation}
\label{4.185} \beta_{2p+1}^+ (z_{2p+1}^*)  \asymp  \frac{1}{p \log
(4p) \, \Omega (4p)}.
\end{equation}

Moreover, there is  no Riesz basis in $L^2 ([0,\pi])$
which consists of antiperiodic  root functions.
\end{Proposition}

\begin{proof}
In view of Criterion~\ref{crit1}, (\ref{4.175}) and (\ref{4.185})
imply that the system of antiperiodic root functions does not contain
Riesz bases. On the other hand, following the proof of
Proposition~\ref{ex1} one can see that (\ref{4.175}) holds.
Therefore, we need to prove (\ref{4.185}) only.

By \eqref{22.9} we have
\begin{equation*}
 \left |\beta^+_{2p+1} (z_{2p+1}^*) - S_1^{21} (2p+1,
z_{2p+1}^*) \right | \leq |V(2n)| + \sum_{k=2}^\infty \left |
S^{21}_k (2p+1,z_{2p+1}^*)\right |.
\end{equation*}
Therefore, from (\ref{4.130}) and (\ref{4.160}) and (\ref{31.41}) it
follows that
\begin{equation}
\label{4.190} \left |\beta^+_{2p+1} (z_{2p+1}^*) - S_1^{21} (2p+1,
z_{2p+1}^*) \right | = o \left ( \frac{1}{p \log (4p) \, \Omega (4p)}
 \right ).
\end{equation}

As in the proof of Proposition~\ref{ex1} one can show that
$$
S^{21}_1 (2p+1, 0) > \frac{1}{8p \log (4p) \,\Omega (4p) \Omega (2)}
$$
and
$$
\left |S^{21}_1 (2p+1, z^*_{2p+1})  - S^{21}_1 (2p+1, 0) \right | =
 o \left ( \frac{1}{p \log (4p) \, \Omega (4p)}  \right ).$$
Therefore, it remains to show that
\begin{equation}
\label{4.201} \left |S^{21}_1 (2p+1, z_{2p+1}^*) \right | =O \left (
\frac{1}{p \,\log (4p) \, \Omega (4p)}\right ).
\end{equation}

By (\ref{31.28}) we have, with $ |z|\leq n/2, $
$$
\left |S^{21}_1 (2p+1, z) \right | \leq  \sum_{j\neq \pm(2p+1)}
\frac{2|V(2p+1-j)V(j+2p+1)|}{|(2p+1)^2- j^2 |} = \sigma_1 + \sigma_2
+ \sigma_3,
$$
where $\sigma_1, \sigma_2 $ and $ \sigma_3 $ are the partial sums of
the above sum, respectively over $\{j<-2p-1\},$ $\{|j|<2p+1\}$ and
 $\{j>2p+1\}.$

First we estimate $\sigma_2. $ Consider the potentials $\tilde{v}$
defined by its Fourier coefficients
$$
\tilde{V}(k) = \begin{cases} V(k)   &  \text{if} \;\;k>0,\\
0   &  \text{if} \;\;k\leq 0.
\end{cases}
$$
From (\ref{4.120}) and (\ref{4.130}) it follows that $\tilde{v} \in
W_\infty (\tilde{\Omega}),$  where $\tilde{\Omega} (k) = \Omega (k)\,
\log k .$  The weight $\tilde{\Omega}$ satisfies the assumptions of
Theorem~\ref{thm1}. Therefore, by (\ref{31.29}) we have
$$
\sigma_2 = S^{21}_1 (\tilde{v}; 2p+1, z_{2p+1}^*) = O \left (
\frac{1}{p \,\log (4p) \, \Omega (4p)}\right ).
$$

The change of variable $j \to -j $ shows that $\sigma_1 = \sigma_3.$
Next we estimate
$$
\sigma_3= \sum_{s=1}^\infty \frac{|V(-2s)|\,V(4p+2+2s)}{s \,
(4p+2+2s)} = \sigma_{3,1} +\sigma_{3,2},
$$
where $\sigma_{3,1}$  and $\sigma_{3,2}$ are respectively the parts
of the above sum over odd $s$ and even $s.$

By (\ref{4.110})--(\ref{4.140}), we have
$$
\begin{aligned}
\sigma_{3,1} &=\frac{|V(-2)|V(4p+4)}{4p+4}+ \sum_{k=1}^\infty
\frac{|V(-4k-2)|V(4p+4k+4)}{(2k+1)(4p+4k+4)} \\
& \leq \frac{1}{p \,\log (4p) \, \Omega (4p)} \left (1+
\sum_{k=1}^\infty \frac{const}{k  \, \Omega
(4k+2)} \right )\\
&=O \left (\frac{1}{p \,\log (4p) \, \Omega (4p)} \right ).
\end{aligned}
$$
Similarly, we obtain
$$
\begin{aligned}
\sigma_{3,2} &=\sum_{k=1}^\infty
\frac{|V(-4k)|V(4p+4k+2)}{2k(4p+4k+2)} \\
& \leq \left (\sum_{k=1}^\infty \frac{const}{k \,\Omega (4k)} \right
)
\frac{1}{p \log (4p)\,\Omega (4p)} \\
&=O \left (\frac{1}{p \,\log (4p) \, \Omega (4p)} \right ).
\end{aligned}
$$
Thus \eqref{4.190} holds, which completes the proof of (\ref{4.185}).
\end{proof}
\bigskip

4. The weighted spaces $W_\infty (\Omega) $ provide a suitable
framework when we study $L^1$-potentials or even potentials that are
finite measures. The next theorem extends the results of
Theorem~\ref{thm2} to a wider class of singular potentials.

\begin{Theorem}
\label{thm22} Let $\Omega = (\Omega (k))_{k \in 2\mathbb{Z}}$ be an
almost submultiplicative weight, and let  $v$ be the potential
defined by its Fourier coefficients $(V(k))_{k \in 2\mathbb{Z}} $
given by
\begin{equation}
\label{51.11} V(k) =|k|^\alpha q(k), \quad  \alpha \in (0, 1/2),
\quad q=(q(k)) \in \ell^\infty (\Omega).
\end{equation}

(a) If $\Delta \subset \mathbb{N}$ is an infinite set such that
\begin{equation}
\label{51.22} |V(\pm 2n)| \, n^{1-2\alpha}\,
 \Omega (2n) \to \infty  \quad \text{as} \;\;
   n \in \Delta,  \;\; n \to \infty,
\end{equation}
then
\begin{equation}
\label{51.23} \beta_n^\pm (v, z) \sim V(\pm 2n) \quad \text{as} \;\;
|z| \leq n/2, \: n \in \Delta, \; n \to \infty.
\end{equation}

(b)  If
 \begin{equation}
\label{51.210}
 \lim_{|k|\to \infty} |q(k)|\,\Omega (k) = 0
\end{equation}
and $\Delta \subset \mathbb{N}$ is an infinite set such that
\begin{equation}
\label{51.220} \exists c>0: \quad |V(\pm 2n)| \,  n^{1-2\alpha}\,
\Omega (2n) \geq c \quad \text{for} \quad n \in \Delta,
\end{equation}
then (\ref{51.23}) holds.

\end{Theorem}

\begin{proof}
We prove \eqref{51.23} for $\beta_n^+$ only since the proof is the
same for $\beta_n^-.$

The following formula (which one can easily verify) will be used:
\begin{equation}
\label{51.100} \sum_{j \neq \pm n} \frac{1}{|n^2 - j^2|^\beta}
\asymp n^{1-2\beta} \quad \text{if} \quad \frac{1}{2} < \beta < 1.
\end{equation}

By \eqref{22.9} we have
\begin{equation} \label{51.30}
|\beta_n^+ (z) - V(2n)| \Omega (2n)\leq \sum_{k=1}^\infty \left
|S^{21}_k (n,z) \right | \Omega (2n).
\end{equation}
Next we estimate the sum on the right.

Let
$$
r(k) = |k|^\alpha q(k), \quad R_m = \sup\{r(k): \; |k| \geq m\}.
$$
In view of (\ref{31.3a}) and \eqref{51.11}, we have
$$
|V(n-j) V(n+j)| \,\Omega (2n) \leq C|n^2 - j^2|^\alpha \,|r(n-j)
r(n+j)| \leq C R_n \|q\|_\Omega.
$$
Therefore, from \eqref{22.15} and \eqref{31.28} it follows that
$$
\left |S^{21}_1 (n,z) \right | \Omega (2n) \leq 2C \sum_{j \neq \pm
n} \frac{R_n \|q\|_\Omega}{|n^2 - j^2|^{1-\alpha}}
$$
which implies, in view of \eqref{51.100},
\begin{equation}
\label{51.31} \left |S^{21}_1 (n,z) \right | = O \left (
n^{-1+2\alpha} \right ) \quad \text{as} \;\;  |z|\leq n/2, \;\; n
\to \infty.
\end{equation}
If \eqref{51.210} holds, then $R_n \to 0, $ so in that case we
obtain
\begin{equation}
\label{51.32} \left |S^{21}_1 (n,z) \right | = o \left (
n^{-1+2\alpha} \right ) \quad \text{as} \;\; |z|\leq n/2, \;\; n \to
\infty.
\end{equation}

Next we estimate $\left |S^{21}_k (n,z) \right | \cdot \Omega (2n)$
for $k \geq 2. $  If  $j_1, \ldots j_k \in (n + 2\mathbb{Z})\setminus
\{\pm n\}, $ then $|n \pm j_s| \geq 2, \;  1 \leq s \leq k,$ so we
have
$$
\begin{aligned}
&\frac{|n-j_1||j_1 - j_2| \cdots |j_{k-1} - j_k||j_k +n|}{|n^2 -
j_1^2|\cdot|n^2 - j_2^2|\cdots |n^2 - j_k^2|}\\
 =& \frac{|j_1 - j_2| }{|n+j_1||n-j_2|} \cdot \frac{|j_2 - j_3|
}{|n+j_2||n-j_3|} \cdots \frac{|j_{k-1} - j_k|
}{|n+j_{k-1}||n-j_k|}\\
= &\left | \frac{1}{n+j_1} + \frac{1}{n-j_2}  \right | \cdot \left |
\frac{1}{n+j_2} + \frac{1}{n-j_3}  \right | \cdots \left |
\frac{1}{n+j_{k-1}} + \frac{1}{n-j_k}  \right | \leq 1.
\end{aligned}
$$
On the other hand, the weight $\Omega $ is almost submultiplicative,
so we have
$$
\Omega (2n) \leq C^k \Omega (n-j_1) \Omega (j_1-j_2) \cdots \Omega
(j_{k-1}-j_k) \Omega (j_k + n).
$$
Therefore by \eqref{22.15}, \eqref{31.28}, \eqref{51.11}, the above
inequalities and \eqref{51.100}, we obtain that
$$
\begin{aligned}
\left |S^{21}_k (n,z) \right | \, \Omega (2n) &\leq (2C)^k
\|q\|_\Omega^{k+1}  \sum_{j_1, \ldots j_k \neq \pm n} \frac{1}{|n^2
- j_1^2|^{1-\alpha}\cdots |n^2 - j_k^2|^{1-\alpha}}\\
&=(2C)^k \|q\|_\Omega^{k+1} \left ( \sum_{j \neq \pm n}
\frac{1}{|n^2 - j^2|^{1-\alpha}} \right )^k \\&=O \left (
\frac{1}{n^{k(1-2\alpha)}} \right ).
\end{aligned}
$$
Now it follows that
\begin{equation}
\label{51.33} \sum_{k=2}^\infty \left |S^{21}_k (n,z) \right | \,
\Omega (2n) = O \left ( n^{2(2\alpha -1)} \right ) \quad \text{as}
\;\; |z|\leq n/2, \;\; n \to \infty.
\end{equation}

By \eqref{51.30}, \eqref{51.31} and \eqref{51.33} we obtain
$$
|\beta_n^+ (z) - V(2n)| \Omega (2n) = O \left ( n^{-1+2\alpha}
\right ) \quad \text{as} \;\;  |z|\leq n/2, \;\; n \to \infty,
$$
so \eqref{51.22} implies \eqref{51.23}.

Moreover, if \eqref{51.210} holds, then \eqref{51.30}, \eqref{51.32}
and \eqref{51.33} imply that
$$
|\beta_n^+ (z) - V(2n)| \Omega (2n) = o \left ( n^{-1+2\alpha}
\right ) \quad \text{as} \;\;  |z|\leq n/2, \;\; n \to \infty.
$$
so if \eqref{51.220} holds then \eqref{51.23} holds also. This
completes the proof.

\end{proof}

\bigskip

5. In this paper, we consider only weighted spaces of $\ell^\infty
$-type. This approach is good in the case of smooth potentials or
even for some classes of singular potentials. But in the case of
singular potentials $v \in H^{-1} (\mathbb{R})$ (see \eqref{i01}) it
is "natural" to work with $\ell^2$-weighted spaces in order to
obtain results similar to Theorem~\ref{thm20} for the whole class of
such potentials. We are going to present such results in another
paper.

\end{document}